\documentclass{amsart}

\usepackage{amssymb, amsmath, amsthm, amstext, romannum}

\usepackage{macros}
\standardsettings

\usepackage{tikz}
\usetikzlibrary{positioning}

\newcommand{\sO}{\mathcal{O}}
\newcommand{\sU}{\mathcal{U}}
\newcommand{\sV}{\mathcal{V}}
\newcommand{\sS}{\mathcal{S}}
\newcommand{\sF}{\mathcal{F}}

\renewcommand{\I}{\text{\Romannum{1}~}}
\renewcommand{\II}{\text{\Romannum{2}~}}
\newcommand{\Is}{\Romannum{1}'s~}
\newcommand{\IIs}{\Romannum{2}'s~}

\newcommand{\rmg}{G^*_{\text{fin}}(\sO,\sO)}
\newcommand{\mg}{G_{\text{fin}}(\sO,\sO)}
\newcommand{\krbg}[1]{G_{#1}(\sO,\sO)}

\newcommand{\rbg}{G_1(\sO,\sO)}

\newcommand{\refine}{\wedge}
\renewcommand{\cl}[1]{\overline{#1}}

\newcommand{\conc}{{}^\smallfrown}
\newcommand{\res}{\restriction}
\newcommand{\lh}{\text{lh}}

\newcommand{\wins}{\text{wins~}}
\newcommand{\zf}{\mathsf{ZF}}
\newcommand{\zfc}{\mathsf{ZFC}}

\renewcommand{\emptyset}{\smallemptyset}

\title{Equivalence of the Rothberger and 2-Rothberger Games for Hausdorff spaces}

\draftfalse

\begin{document}
\authorlogan
\authorlior
\authornathaniel
\authorsteve

\pagenumbering{arabic}

\begin{abstract}
We prove that in any Hausdorff space, the Rothberger game is
equivalent to the $k$-Rothberger game, i.e. the game in which player \II
chooses $k$ open sets in each move. This result follows from a more
general theorem in which we show these games are equivalent to
a game we call the restricted Menger game. In this game
\I knows immediately in advance of playing
each open cover how many open sets \II will choose from that open cover. 
This result illuminates the relationship between the
Rothberger and Menger games in Hausdorff spaces. The equivalence of these
games answers a question posed by Aurichi, Bella, and Dias \cite{AurBelDias2016},
at least in the context of Hausdorff spaces.

\end{abstract}

\thanks{The authors wish to thank the organizers of the conference
{\em Frontiers of Selection Principles} held at Cardinal Stefan Wysz\'{n}ski
University in Warsaw in August 2017 where the results of this paper were presented.}

\maketitle

\section{Introduction}  \label{sec:intro}
Let $X$ be a topological space. Let $\sO$ denote the collection of open covers of $X$.
The {\em Menger game} on $X$ is the two player game
where at each round $n$ of the game player \I first plays an open cover $\sU_n\in \sO$ of $X$,
and player \II responds by playing a finite subset $U_n^0,\dots, U_n^{k_n-1}$ of $\sU_n$.
Player \II wins the run of the game if $X=\bigcup_n \bigcup_{i<k_m} U_n^i$. We denote
the Menger game by $\mg$. The notation reflects the facts that \I is playing from $\sO$,
\II is trying to build an element of $\sO$, and \II is picking a finite subset from
\Is moves at each round. 
The {\em Rothberger game} \cite{Rothberger38}, $\rbg$,  on $X$ is the game where player \I
plays at round $n$ an open cover $\sU_n\in \sO$ and player \II plays a single $U_n \in \sU_n$. Again, player \II
wins the run of the game iff $X=\bigcup_n U_n$. The $k$-Rothberger game $\krbg{k}$ is the variation
of the Rothberger game where player \II plays $k$ sets from \Is cover at each round. 
A natural extension of this is the game $\krbg{f}$ where $f \colon \omega\to \omega^{>0}$. In this game,
at each round $n$ player \II plays $f(n)$ sets from player \Is move $\sU_n$. A still further extension
of the games is the {\em restricted Menger game}  $\rmg$, which we define precisely below,
where player \II decides at the start of each round $n$ how many sets he will get to choose from \Is
play $\sU_n$. It is clear that 
\[
\II \wins \rmg \Rightarrow \forall f\ \II \wins \krbg{f} \Rightarrow \forall k \ \II \wins \krbg{k}
\Rightarrow \II \wins \rbg
\]

Our main result, Theorem~\ref{mt}, is that for all $T_2$ spaces $X$, the above games are all
equivalent. Recall two games are said to be equivalent if whenever one of the players has
a winning strategy in one of the games, then that same player has a winning strategy
in the other game. We note that the equivalence of the above games for arbitrary spaces
is no stronger than the equivalence for $T_0$ spaces (by considering the
$T_0$ quotient of an arbitrary space).
On the other hand, it is well known that the full Menger game
$\mg$ is not equivalent to the above mentioned games. For example, player \II wins the
Menger game on $\R$, or any $\sigma$-compact space, while \I has a winning strategy in
$\rbg$ on $\R$ (\I can easily play to ensure that $\lambda(\bigcup U_n)<\epsilon$
for any given $\epsilon >0$).

The games mentioned above are closely related to selection principles on the space $X$.
These types of covering games and selection principles were extensively studied by Scheepers
and others, see for example \cite{Scheepers96}, \cite{SakaiAndScheepers04}. 
Recall that $X$ has the {\em Menger property}, denoted $\sS_{\text{fin}}(\sO,\sO)$,
if whenever $\{ \sU_n\}_{n \in \omega}$ is a sequence
of open covers of $X$, then there is a sequence $\{\sF_n \}_{n \in \omega}$, where each
$\sF_n$ is a finite subset of $\sU_n$, such that $X= \bigcup_n \cup \sF_n$.
Similarly, $X$ has the {\em Rothberger property}, denoted $\sS_1(\sO,\sO)$, if 
whenever $\{ \sU_n\}_{n \in \omega}$ is a sequence
of open covers of $X$, then there is a sequence $U_n \in \sU_n$ such that $X=\bigcup_n U_n$.
There are two theorems which relate the games with the corresponding selection principles.
One theorem, due to Hurewicz \cite{Hurewicz25} (see also \cite{Scheepers96}),
says that for any space $X$ the selection
principle $\sS_{\text{fin}}(\sO,\sO)$ (i.e., $X$ having the Menger property) is equivalent
to \I not having a winning strategy in $\mg$. Another theorem, due to Pawlikowski \cite{Pawlikowski1994},
says that for any space $X$ the selection property $\sS_1(\sO,\sO)$ (i.e., $X$
having the Rothberger property) is equivalent to \I not having a winning strategy in
$\rbg$. The equivalence of $\sS_k(\sO,\sO)$ (where $k\in\omega$)
and $\sS_1(\sO,\sO)$ was shown in \cite{GFandTM95} and noted by the authors of \cite{AurBelDias2016}.

The Rothberger game $\rbg$, for any space $X$, has a dual version called the
{\em point-open game}. In this game, \I plays at each round $n$ a point $x_n\in X$,
and \II then plays an open set $U_n$ with $x_n \in U_n$. Player \I wins the run
of the game iff $X= \bigcup_n U_n$. A theorem of Galvin \cite{Galvin78} says that (for any $X$)
these games are {\em dual}, that is, one of the players has a winning strategy in one of the games iff
the other player has a winning strategy in the other game. A natural variation of the point-open
game is the {\em finite-open game}, where \I plays at each round $n$ a finite set $F_n \subseteq X$,
and \II plays an open set $U_n$ with $F_n \subseteq U_n$. Player \I again wins the run iff
$X=\bigcup_n U_n$. It is easy to see that for any $X$ that the point-open game is equivalent
to the finite-open game.

Using these dual games (specifically the finite-open game)
simplifies the presentation of our main result. This observation was noted by
R.\ Dias, whom we thank.

\section{Equivalence of Restricted Menger and Rothberger Games} \label{sec:mt}

We define a variation of the Menger game which we call the {\em
restricted Menger game}, denoted by $\rmg$. The rounds of this game
are as in the Menger game except that at the start of round $n$
player \II will make an initial move, which must be a positive integer
$k_n$, which is a declaration of how many open sets \II intends to select this round. As in the Menger game, \I will then play an open cover
$\sU_n\in \sO$, and \II will then respond by choosing $k_n$ of the
sets from $\sU_n$, which we denote $U_n^0,\dots,U_n^{k_n-1}$.  Player
\II wins the run of the game iff $X= \bigcup_n \bigcup_{i < k_n}
U_n^i$.

\begin{figure}[h]
\begin{tikzpicture} [node distance=0.125cm]

\foreach \y in {1}{
\node (begin\y) at (0,-\numexpr2*\y\relax) {};
\node (s\y) [right=of begin\y] {};
\node (G\y) [left=0.5cm of s\y] {$G_{\text{fin}}^*(\sO, \sO)$};
\node (G\y{}I0) [above=of s\y] {I};
\node (G\y{}II0) [below=0.3cm of G\y{}I0] {II};
\node (G\y{}I1) [above right=0.25cm and 0.125cm of G\y{}II0] {};
\node (G\y{}II1) [below right=0.25cm and 0.125cm of G\y{}I0] {$k_0$};
\foreach \x in {2, 3}
{
\node (G\y{}I\x) [above right=0.25cm and 0.125cm of G\y{}II\the\numexpr\x-1\relax] {$\sU_{\the\numexpr\x-2\relax}$};
\node (G\y{}IIos\x) [below right=0.25cm and 0.125cm of G\y{}I\the\numexpr\x\relax] {$\{U_{\the\numexpr\x-2\relax}^i\}_{i<k_\the\numexpr\x-2\relax}$};
\node (G\y{}II\x) [right=0.25cm of G\y{}IIos\the\numexpr\x\relax] {$k_{\the\numexpr\x-1\relax}$};
\draw [dashed] (G\y{}IIos\x)+(0.95 cm, -0.5 cm)  --+(0.95, 1 cm);
}
\node (G\y{}I4) [above right=0.25cm and 0.125cm of G\y{}II\the\numexpr4-1\relax] {$\sU_{\the\numexpr4-2\relax}$};
\node (G\y{}IIos4) [below right=0.25cm and 0.125cm of G\y{}I\the\numexpr4\relax] {$\{U_{\the\numexpr4-2\relax}^i\}_{i<k_\the\numexpr4-2\relax}$};
\node (end) [above right=0.05cm and 0.25cm of G\y{}IIos4] {$\dots$};
}

\end{tikzpicture}

\caption{An illustration of the game $G_{\text{fin}}^*(\sO, \sO)$.}

\end{figure}
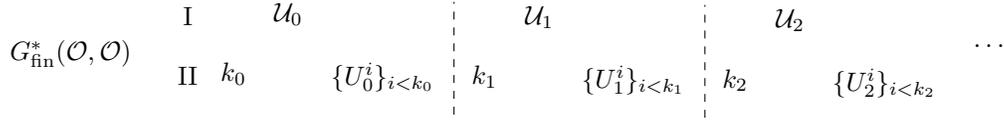

\begin{theorem} \label{mt}
Let $X$ be a $T_2$ space. Then the restricted Menger game $\rmg$ is equivalent to the Rothberger game $\rbg$. 
\end{theorem}

\begin{proof}
It is clear that if \I wins $\rmg$ then  \I wins $\rbg$. It is also clear that if \II wins  $\rbg$, then \II wins $\rmg$.

If \I wins $\rbg$, then by \cite{Pawlikowski1994}, $X$ does not satisfy the selection
principle $\sS_1(\sO,\sO)$. Thus, there is a sequence $\{\sV_n\}$ of
open covers of $X$ such that there is no sequence $V_n \in \sV_n$ with
$X=\bigcup_n V_n$. Then \I has a winning strategy in $\rmg$ by playing
as follows.  If \II first plays the integer $k_0$, then \I plays the
common refinement $\sU_0=\sV_0 \refine \cdots \refine
\sV_{k_0-1}$. \II will end the round by picking $k_0$ of the sets
$U_0^0, \dots, U_0^{k_0-1}$ from $\sU_0$. Player \I continues in this
manner. Because of the refining property of the $\sU_n$, there is a
sequence $V_n \in \sV_n$ with $\bigcup_n \bigcup_n^i U_n^i \subseteq
\bigcup_n V_n $. Since $\bigcup_n V_n \neq X$, \I has won this run of
$\rmg$.

Assume now that \II has a winning strategy $\tau$ in $\rmg$. We let
$\tau(\sU_0,\dots,\sU_n)$ denote the response of $\tau$ when \I plays open covers
$\sU_0,\dots,\sU_n$ (we are suppressing \IIs moves according to $\tau$ in this notation).
So, $\tau(\sU_0,\dots,\sU_n)$ is a finite subset of $\sU_n$. We let $\tau'(\sU_0,\dots,\sU_n)$
denote the integer that $\tau$ plays at the start of the next round, immediately after
$\tau(\sU_0,\dots,\sU_n)$ was played. 
By $\cup \tau(\sU_0,\dots,\sU_n)$ we mean the union of the (finitely many) open sets
in $\tau(\sU_0,\dots,\sU_n)$. Note that according to this notation
$|\tau(\sU_0,\dots,\sU_n)| = \tau'(\sU_0,\dots,\sU_{n-1})$.

We define a strategy $\sigma$ for \I in the finite open game on $X$.
We begin by explicitly describing $\sigma$ on the first round.
Let $k_\emptyset$ be $\tau$'s first (integer) move in $\rmg$.
Define
\[
C_\emptyset=  \bigcap_{\sU\in \sO} \cl {\cup \tau(\sU)}.
\]

The next Lemma is the only point in the proof where we use the assumption that $X$ is $T_2$.

\begin{lemma} \label{clem}
$|C_\emptyset| \leq k_\emptyset$. 
\end{lemma}
\begin{subproof}
Suppose towards a contradiction that $x_0, \dots,x_{k_\emptyset}$ are $k_\emptyset+1$ distinct points in
$C_\emptyset$. Since $X$ is $T_2$, there are open sets $U_0,\dots,U_{k_\emptyset}$ in $X$ with
$x_i \in U_i$ for all $i \leq k_\emptyset$ and with the $\{U_i\}$ pairwise disjoint. For each
$x \in X \setminus \{ x_i\}_{i \leq k_\emptyset}$ let $U_x$ be an open set containing $x$ such that
$U_x$ is disjoint from a neighborhood of $\{ x_i\}_{i \leq k_\emptyset}$ (using $T_2$ again). 
Let $\sU= \{ U_x\colon x \notin \{ x_i\}_{i \leq k_\emptyset} \} \cup \{ U_i\}_{i \leq k_\emptyset}$,
so $\sU$ is an open cover of $X$. $\tau(\sU)$ consists of  $k_\emptyset$ of the sets from $\sU$.
There is an $i \leq k_\emptyset$ such that $U_i \notin \tau(\sU)$. Then $x_i\notin \cl{ \cup \tau(\sU)}$,
a contradiction to $x_i \in C_\emptyset$. 
\end{subproof}

Then let $\sigma$'s first move in the finite open game be
$C_\emptyset$. Say \II responds with $V_0$. Before we continue, we
need to define some auxiliary sets which correspond to the position
$\{C_\emptyset, V_0\}$.  If $V_0$ was legal, then we note that $X
\setminus V_0 \subseteq X \setminus C_\emptyset$, and thus for each $x
\in X \setminus V_0$, there is some $\sU \in \sO$ such that $x \in X
\setminus \cl {\cup \tau(\sU)}$.  These sets form an open cover of $X
\setminus V_0$, which is a closed subspace of $X$, and thus is
Lindel\"{o}f, and so we fix $\{\sU_{(m)}(V_0)\}_{m \in
  \omega}=\{\sU_{(m)}\}_{m \in \omega}$ such that $\{ X\setminus \cl
{\cup \tau(\sU_{(m)}}\}_{m \in \omega}$ is a cover of $X \setminus
V_0$.

To define $\sigma$ in subsequent rounds, we need to dovetail various
moves on subsequences, using the previously defined open covers
$\sU_s$ for $s \in \omega^{<\omega}$, and for this purpose we fix any
bijection $\varphi \colon \omega^{<\omega}\to \omega$ with the
property that if $s \subseteq t$ then $\varphi(s)\leq \varphi(t)$.
For $s \in \omega^{<\omega}$ we let $\lh(s)$ denote the length of $s$.
Now in general, suppose we are at round $n$ in the finite open game,
and the moves $C_0, V_0, \dots, C_{n-1}, V_{n-1}$ have been played,
where $|C_i|=k_i$ for $i<n$.  Assume in addition that for each $j<n$
we have defined open covers $\sU_{\varphi^{-1}(j)\conc m}$ for all $m
\in \omega$ (which depend on the $V_j$ played thus far). Furthermore,
assume that the $C_j, V_j, \sU_{\varphi^{-1}(j)\conc m}$ for $j<n$
satisfy the following. Let $s=\varphi^{-1}(j)$, then:

\begin{enumerate}
\item \label{cp} $C_j= \bigcap_{\sU\in \sO} \cl {\cup \tau(\sU_{s\res
    1}, \sU_{s \res 2},\dots,\sU_{s\res\lh(s)}, \sU)}$.
\item \label{up} $\{ X\setminus \cl {\cup \tau(\sU_{s\res 1}, \sU_{s
    \res 2},\dots,\sU_{s\res\lh(s)}, \sU_{s\conc m}) }\}_{m \in
  \omega}$ is a cover of $X\setminus\bigcup_{i \leq \lh(s)}
  V_{\varphi(s \res i)}$.
\end{enumerate}

\noindent
Note that property~(\ref{up}) for $j$ is possible since the space
$X\setminus\bigcup_{i \leq \lh(s)} V_{\varphi(s \res i)}$ is
Lindel\"{o}f and $X\setminus\bigcup_{i \leq \lh(s)} V_{\varphi(s \res
  i)} \subseteq X\setminus\bigcup_{i \leq \lh(s)} C_{\varphi(s \res
  i)}$, and using property~(\ref{cp}) for the $C_i$ for $i \leq j$.

We define $\sigma$'s response to this position, and the necessary sets
$\sU_{t \conc m}$, in a similar manner to the base step.  Let
$t=\varphi^{-1}(n)$ and define $\sigma$'s response to be
\[
C_n=\bigcap_{\sU\in \sO}
\cl {\cup \tau(\sU_{t\res 1}, \sU_{t \res 2},\dots,\sU_{t\res\lh(t)}, \sU)},
\]
which clearly maintains
property~(\ref{cp}). Note also that $C_n$ is finite, and in fact has size at most
$|C_n| \leq \tau'(\sU_{t \res 1},\dots, \sU_{t \res \lh(t)-1})$, by the same proof
of Lemma~\ref{clem}.

Similarly to the base step, define $\{\sU_{t
  \conc m}\}_{m \in \omega}$ to be a countable collection of open
covers such that $\{ X\setminus \cl {\cup \tau(\sU_{t\res 1}, \sU_{t
    \res 2},\dots,\sU_{t\res\lh(t)}, \sU_{t\conc m}) }\}_{m \in
  \omega}$ covers $X\setminus\bigcup_{i \leq \lh(t)} V_{\varphi(t \res
  i)}$.  Of course, this uses the fact that $X\setminus\bigcup_{i \leq
  \lh(t)} V_{\varphi(t \res i)}$ is Lindel\"{o}f and that it is
contained in $X\setminus C_{n}$.  This completes the definition of
$\sigma$.  To show that $\sigma$ is winning, we suppose that $C_0,
V_0, C_1, V_1, \dots$ is a full run of the finite open game which is
consistent with $\sigma$.  Note that since this run is consistent with
$\sigma$, we can recover the tree of open covers $\{\sU_s\}_{s \in
\omega^{<\omega}}$ associated to this run which satisfies the
properties~(\ref{cp}) and (\ref{up}) above.  Suppose that $X \neq
\bigcup_n V_n$, and let $x \in X \setminus\bigcup_n V_n$.  In
particular, $x \in X \setminus V_0$.  Now we use property~(\ref{up})
to obtain $i_0$ such that $x \not \in \cl {\cup \tau(\sU_{(i_0)})}$.
In general, supposing we have $i_0, i_1, \dots, i_{n-1}$ where $x \not
\in \cl {\cup \tau(\sU_{(i_0)}, \dots\sU_{(i_0, \dots, i_k)})}$ for
any $k < n$, then use the fact that $x \in X \setminus \bigcup_{s
  \subseteq \varphi^{-1}(n)} V_{\varphi(s)}$ and property~(\ref{up})
to obtain $i_n$ so that $x \not \in \cl {\cup \tau(\sU_{(i_0)},
  \dots,\sU_{(i_0, \dots, i_{n-1})}, \sU_{(i_0, \dots, i_n)})}$.
This builds a branch through the tree of open covers $\{\sU_s\}_{s \in
  \omega^{<\omega}}$, associated to this run, which has the property
that $x$ is not in any of the closures of $\tau$'s moves in response
to this branch.  This contradicts the assumption that $\tau$ was a
winning strategy.
\end{proof}

\begin{corollary}
For any $T_2$ space $X$ and any $f \colon \omega \to \omega$, the games
$G_1(\sO,\sO)$ and $G_f(\sO,\sO)$ are equivalent.
\end{corollary}

In particular, we have the following corollary which answers Problem~4.5 of 
\cite{AurBelDias2016} for $T_2$ spaces.

\begin{corollary}
For any $T_2$ space $X$ and any $n \in \omega$, the games $G_1(\sO,\sO)$
and $G_n(\sO,\sO)$ are equivalent.
\end{corollary}

\section{Open Questions} \label{sec:oq}

A natural question is whether we can drop the assumption that $X$ is $T_2$ from the hypothesis
of Theorem~\ref{mt}. In fact, the authors of \cite{AurBelDias2016} originally
asked if for any topological space the games
$\krbg{1}$ and $\krbg{2}$ are equivalent. Our Theorem~\ref{mt} shows these games are
equivalent for any $T_2$ space, but the $T_2$ assumption seems necessary for the argument.
We are not aware of any space (with no assumptions on the space) for which these games are not equivalent.
Since the determinacy of these games is not guaranteed in $\zf$, it is possible even that the equivalence
for arbitrary spaces is independent of $\zf$.

\begin{question} \label{qa}
Can we weaken the hypotheses of Theorem~\ref{mt} from $T_2$ to $T_1$, or even remove it entirely?
That is, can we prove in $\zfc$ that the games $\krbg{1}$ and $\krbg{2}$ are equivalent
for any space $X$?
\end{question}

One possibility for a negative answer to Question~\ref{qa} would be to construct in $\zfc$
a space for which the games are not equivalent (in this case the game $\krbg{1}$ is not
determined, and \II must win the other game). It is also possible that the existence of a
space for which the games are not equivalent is independent of $\zfc$. So we ask:

\begin{question}
Is it consistent with $\zfc$ that there is a space $X$ for which the games
$\krbg{1}$ and $\krbg{2}$ are not equivalent. Is the existence of such a space
consistent with $\zf$? In particular are the games equivalent in models of determinacy?
\end{question}

\bibliographystyle{amsplain}
\bibliography{selection}
\end{document}